\documentclass[ECP]{ejpecp}
\usepackage[latin1]{inputenc}

\usepackage{amsthm,amsmath,amssymb,amsfonts}

\usepackage{array}
\usepackage{enumerate}
\usepackage{bm}

\SHORTTITLE{Point process}
\TITLE{Point processes in a metric space\thanks{The author gratefully acknowledges funding by IAP research network grant nr. P7/06 of the Belgian government (Belgian Science Policy).}}

\AUTHORS{
Yuwei~Zhao\footnote{Universit\'e catholique de Louvain, Belgium. \EMAIL{yuwei.zhao@uclouvain.be}}}


\AMSSUBJ{60G55; 60G70;60G57} 

\SUBMITTED{NA} 
\ACCEPTED{NA} 

\VOLUME{0}
\YEAR{2016}
\PAPERNUM{0}
\DOI{10.1214/YY-TN}

\ABSTRACT{
As a useful and elegant tool of extreme value theory, the study of
point processes on a metric space is important and necessary for the
analyses of heavy-tailed functional data. This paper focuses on the
definition and properties of such point processes. A complete
convergence result for a regularly varying iid sequence in a metric
space is proved as an example of the application in extreme value theory.   
}

\renewcommand{\le}{\leqslant}
\renewcommand{\ge}{\geqslant}

\begin{document}
\section{Introduction}
\label{sec:introduction}

Point process theory is considered as a useful tool in the analyses of
extremal events of stochastic processes due to its close connection to
the concept of regular variation, which plays an important role in
modelling extremal events. For example, for an iid copies
$(X_i)_{i\ge 1}$ of a random element $X$, it is well known that
regular variation of $X$ is 
equivalent to the convergence of point processes
\begin{equation*}
N_n=\sum_{i=1}^n\delta_{(i/n,X_i/a_{n})}
\end{equation*}
towards a suitable Poisson point process. This result is referred to
as the complete convergence result; see for example \cite{RES87}. A
large number of literatures focus on the extension of this result and
its applications; see \cite{DH95,
  DM98,Basrak:2012br,Basrak:2016gg}, to name a few. There are also a
large amount of literatures on the applications of regular variation
in extreme value theory, such as
\cite{EKM97,Resnick:2007jl,BS09,Davis:2013db}, to name a few. 

The classical formulation of regular variation for measures on a space
such as $\mathbb{R}$ and $\mathbb{R}^d$ uses vague convergence; see
for instance \cite{RES87,EKM97} for more details. A
generalization of vague convergence is $w^{\#}$-convergence introduced
by \cite{Daley:2003tf, Daley:2008kf} in which point process theory
defined via $w^{\#}$-convergence is well developed. Consider a
complete and separable space 
$(\mathbb{S},d)$ with a point $0_{\mathbb{S}}\in \mathbb{S}$. Let
$\mathbb{M}_b(\mathbb{S})$ denote the class of totally finite
measures on $\mathcal{S}=\mathcal{B}(\mathbb{S})$, let
$\mathbb{M}_+(\mathbb{S})$ denote the class of measures finite on
compact Borel sets in $\mathcal{S}$ and let
$\mathbb{M}_{\#}(\mathbb{S})$ denote the class of measures finite on
bounded Borel sets in $\mathcal{S}$. One can define weak
convergence, vague convergence, $w^{\#}$-convergence in the
corresponding class of measures $\mathbb{M}_b(\mathbb{S})$,
$\mathbb{M}_+(\mathbb{S})$ and $\mathbb{M}_{\#}(\mathbb{S})$. It
should be noted as explained in \cite{Daley:2003tf} that
$w^{\#}$-convergence coincides with vague convergence when
$\mathbb{S}$ is locally compact. 
Borel sets bounded away from the origin are of interest 
in the field of extreme value theory and similar to the ideas of
$w^{\#}$-convergence, let
$\mathbb{M}_0=\mathbb{M}_{0}(\mathbb{S})$ denote the class of the measures finite on
complements of neighborhoods of the point $0_{\mathbb{S}}$ and regular
variation of measures in $\mathbb{M}_0(\mathbb{S})$ is defined via the
concept of $\mathbb{M}_{0}$-convergence; see \cite{HL06} for
details. In several aspects, $\mathbb{M}_0$-convergence has a lot of
similarities to $w^{\#}$-convergence and it allows us to use the
results in \cite{Daley:2003tf,Daley:2008kf} to construct point process
theory via $\mathbb{M}_0$-convergence.

The effort of generalizing the concept of $\mathbb{M}_0$-convergence
focuses on allowing the introduction of hidden regular variation into
the framework and studying simultaneously regular-variation properties
at different scales. The basic idea is to removing a closed set
$\mathbb{C}$ instead of one point $0_{\mathbb{S}}$, which leads to
$\mathbb{M}_{\mathbb{O}}$-convergence where
$\mathbb{O}=\mathbb{S}\setminus \mathbb{C}$; see
\cite{Lindskog:2014ia} for a formal definition. Moreover, if
$\mathbb{C}$ is a closed cone, $\mathbb{M}_{\mathbb{O}}$-convergence
shares many properties with $\mathbb{M}_0$-convergence;
see \cite{HL06,Lindskog:2014ia}. For applications of the concept of
$\mathbb{M}_{\mathbb{O}}$-convergence, it is necessary to develop
point process theory as a useful and elegant tool. The basic idea of
this paper is to define point processes via
$\mathbb{M}_{\mathbb{O}}$-convergence and then to use them to analyse
the heavy-tailed processes in a metric space. 

The paper is arranged as follows. Section~\ref{sec:m0} gives an
introduction to the spaces of measures on a complete and separable
space, and $\mathbb{M}_{\mathbb{O}}$-convergence. Definitions and
properties of random measures and point processes are given in 
Section~\ref{sec:rand-meas-lapl}. Poisson processes play an important
role in applications of point processes, which is discussed in
Section~\ref{sec:poisson-processes}. The complete convergence result
of point processes is proved under $\mathbb{M}_{\mathbb{O}}$-convergence in
Section~\ref{sec:regular-variation}. Section~\ref{sec:discussions}
provides brief discussions on the choice of $\mathbb{O}$ and spaces of
measures. Section~\ref{sec:proof} contains technical proofs
of theorems in the previous sections.

\section{The spaces of measures}
\label{sec:m0}
Let $(\mathbb{S},d)$ be a complete and separable space and let
$\mathcal{S}=\mathcal{B}(\mathbb{S})$ denote the Borel $\sigma$-field
on $\mathbb{S}$, which is generated by open balls $B_{r}(x)=\{y\in
\mathbb{S}:d(x,y)<r\}$ for $x\in \mathbb{S}$. Assume that
$\mathbb{C}\in \mathcal{S}$ is a closed subset and let
$\mathbb{O}=\mathbb{S}\setminus \mathbb{C}$. The $\sigma$-algebra
$\mathcal{S}_{\mathbb{O}}=\{A\in \mathcal{S}:A\subset
\mathbb{O}\}$. Let
$\mathcal{C}_{\mathbb{O}}=\mathcal{C}_{\mathbb{O}}(\mathbb{S})$ be a
collection of real-valued, non-negative, bounded continuous functions
$f$ on $\mathbb{O}$ vanishing on $\mathbb{C}^r=\{x\in
\mathbb{S}:d(x,\mathbb{C})=\inf_{y\in \mathbb{C}}d(x,y)<r\}$ for some
$r>0$. We say that a set $A\in \mathcal{S}_{\mathbb{O}}$ is bounded away from
$\mathbb{C}$ if $A\subset \mathbb{S}\setminus \mathbb{C}^r$ for some
$r>0$. To define regular variation of measures in
$\mathbb{M}_{\mathbb{O}}$, $\mathbb{S}$ is assumed to be equipped with
scalar multiplication; see Section 3.1, \cite{Lindskog:2014ia}.
\begin{definition}\label{def:scalar}
  A scalar multiplication on $\mathbb{S}$ is a map $[0,\infty)\times
  \mathbb{S}\to \mathbb{S}: (\lambda,x)\to \lambda x$ satisfying the
  following properties:
  \begin{enumerate}[(i)]
  \item $\lambda_1(\lambda_2 x)=(\lambda_1 \lambda_2)x$ for all
    $\lambda_1,\lambda_2\in [0,\infty)$ and $x\in \mathbb{S}$;
\item\label{item:6} $1x=x$ for $x\in \mathbb{S}$;
\item\label{item:7} the map is continuous with respect to the product
  topology;
\item\label{item:8} if $x\in \mathbb{O}$ and if $0\le
  \lambda_1<\lambda_2$, then $d(\lambda_1 x,\mathbb{C})<d(\lambda_2
  x,\mathbb{C})$. 
  \end{enumerate}
\end{definition}
Assume additionally that $\mathbb{C}$ is a cone, that is, $\lambda
\mathbb{C}=\mathbb{C}$ for $\lambda\in (0,\infty)$. Let $x\in \mathbb{O}$. For any $\lambda\in [0,\infty)$,
we have $\lambda (0x)=(\lambda 0)x=0x$ by (i) in
Definition~\ref{def:scalar}. It follows that
$d(\lambda_1(0x),\mathbb{C})=d(0x,\mathbb{C})=d(\lambda_2(0x),\mathbb{C})$
for all $\lambda_1,\lambda_2\in (0,\infty)$. The condition (iv) in
Definition~\ref{def:scalar} implies that $0x\in \mathbb{C}$. If $x\in
\mathbb{C}$, $0x=0(0x)$ and the above argument implies that $0x\in
\mathbb{C}$. So $0\mathbb{C}\subset \mathbb{C}$ and $\mathbb{O}$ is an open cone.  

We always assume that the metric space $(\mathbb{S},d, \mathbb{C})$ is equipped
with a scalar multiplication defined in Definition~\ref{def:scalar}
and $\mathbb{C}$ is a closed cone. For
simplicity, we will write the space $(\mathbb{S},d,\mathbb{C})$ as
$(\mathbb{S},d)$ or $\mathbb{S}$. 
\subsection{The space $\mathbb{M}_{\mathbb{O}}$ and
  $\mathbb{M}_{\mathbb{O}}$-convergence} 
\label{sec:space-mathbbm_m}

Let $\mathbb{M}_{\mathbb{O}}=\mathbb{M}_{\mathbb{O}}(\mathbb{S})$ be the space
of Borel measures on $\mathcal{S}_{\mathbb{O}}$ that
are bounded on complements of $\mathbb{C}^r$, $r>0$. The convergence
$\mu_n\to \mu$ in $\mathbb{M}_{\mathbb{O}}$ or $\mu_n\overset{M}{\to}
\mu$ holds if and only if $\int
f\,d\mu_n\to \int f\,d\mu$ for all $f\in
\mathcal{C}_{\mathbb{O}}$. Versions of the Portmanteau and continuous
mapping theorem for $\mathbb{M}_{\mathbb{O}}$-convergence are stated
as Theorem 2.1 and Theorem 2.3, respectively, in
\cite{Lindskog:2014ia}. By choosing the metric
\begin{eqnarray*}
  d_{\mathbb{M}_{\mathbb{O}}}(\mu,\upsilon)=
  \int_0^{\infty}e^{-r}\frac{p_r(\mu^{(r)},\upsilon^{(r)})}{1+p_r(\mu^{(r)},\upsilon^{(r)})}\,,
  \quad \mu,\upsilon\in \mathbb{M}_\mathbb{O}\,,
\end{eqnarray*}
where $\mu^{(r)}$, $\upsilon^{(r)}$ are the finite restriction of
$\mu,\upsilon$ to $\mathbb{S}\setminus \mathbb{C}^r$ and $p_r$ is the
Prohorov metric on the space of finite Borel measures on
$\mathcal{B}(\mathbb{S}\setminus \mathbb{C}^r)$. It is shown in
\cite{Lindskog:2014ia} that
$d_{\mathbb{M}_{\mathbb{O}}}(\mu,\upsilon)\to 0$ as $n\to \infty$ if
and only if $\mu_n\to \mu$ in $\mathbb{M}_{\mathbb{O}}$ and the space
$(\mathbb{M}_{\mathbb{O}},d_{\mathbb{M}_{\mathbb{O}}})$ is complete
and separable. Denote $\mathcal{M}_{\mathbb{O}}$ as the
$\sigma$-algebra of $\mathbb{M}_{\mathbb{O}}$, which is generated by
the neighborhoods of $\mu\in \mathbb{M}_{\mathbb{O}}$
\begin{eqnarray*}
 \Big\{\upsilon\in \mathbb{M}_{\mathbb{O}}:\Big|\int f_i\,d\upsilon-\int
  f_i\,d\mu\Big|<\varepsilon,\; i=1,\ldots,k\Big\}\,,
\end{eqnarray*}
where $\varepsilon>0$ and $f_i\in \mathcal{C}_{\mathbb{O}}$ for
$i=1,\ldots,k$. 
\begin{proposition}\label{prop:algebra}
  The Borel $\sigma$-algebra $\mathcal{M}_{\mathbb{O}}$ is the
  smallest $\sigma$-algebra with respect to which the mappings
  $\Phi_{A}:\mathbb{M}_{\mathbb{O}}\mapsto \mathbb{R}\cup \{\pm \infty\}$ given by
\begin{eqnarray*}
 \Phi_A(\mu)=\mu(A)\,, 
\end{eqnarray*}
are measurable for all sets $A$ in a $\pi$-system $\mathcal{D}$ generating
$\mathcal{S}_{\mathbb{O}}$ and in particular for the sets $A\in
\mathcal{S}_{\mathbb{O}}$ bounded away from $\mathbb{C}$. 
\end{proposition}
The proof is technical and is put in Section~\ref{sec:proof}. 
\subsection{The space $\mathbb{N}_{\mathbb{O}}$}
\label{sec:space-n}

Let $\mathbb{N}_{\mathbb{O}}$ be the space of all measures $N\in
\mathbb{M}_{\mathbb{O}}$ satisfying that for each $r>0$, $N(A\setminus
\mathbb{C}^r)$ is a non-negative integer for all $A\in \mathcal{S}$. We
call the measure $N$ a counting measure for short. For $x\in S_0$ and
$\mu\in \mathbb{M}_{\mathbb{O}}$, we say that the measure $\mu$ has an atom
$x$ if $\mu(\{x\})>0$. A measure with only atoms is {\em purely
  atomic}, while a {\em diffuse} measure has no atom. We use
$\delta_x$ to denote Dirac measure at $x\in \mathbb{O}$, defined on $A\in
\mathcal{S}_{\mathbb{O}}$ by
\begin{eqnarray*}
\delta_x(A)=
  \begin{cases}
    1\,, & x\in A\,, \\
0\,, & x\notin A\,.
  \end{cases}
\end{eqnarray*}
The following lemma shows that the counting measure is purely atomic. 
\begin{lemma}\label{lem:decomposition}
Assume that the measure $\mu\in\mathbb{M}_{\mathbb{O}}$. 
  \begin{enumerate}[(i)]
  \item The measure $\mu$ is uniquely decomposable as
    $\mu=\mu_a+\mu_d$, where 
\begin{equation}
\label{eq:decomposition}
\mu_a=\sum_{i=1}^{\infty}\kappa_i \delta_{x_i}
\end{equation}
is a purely
atomic measure, uniquely determined by a countable set
$\{(x_i,\kappa_i)\}\subset \mathbb{O}\times (0,\infty)$, and $\mu_d$ ia a
diffuse measure. 
\item\label{item:1} A measure $N\in \mathbb{M}_{\mathbb{O}}$ is a counting
  measure if and only if (1) its diffuse component is null, (2) all
$\kappa_i$ in \eqref{eq:decomposition} are positive integers, and (3)
the set $\{x_i\}$ defined in \eqref{eq:decomposition} is a countable
set with at most finite many $x_i$ in any set $\mathbb{S}\setminus
\mathbb{C}^r$ with $r>0$.
  \end{enumerate}
\end{lemma}
\begin{proof}
 Let $r_j=1/j$, $j=1,2,\ldots$. Let
 $\mathbb{O}^{(1)}=\mathbb{S}\setminus \mathbb{C}^{r_1}$ and
 $\mathbb{O}^{(j+1)}=\mathbb{C}^{r_j}\setminus \mathbb{C}^{r_{j+1}}$,
 $j=1,2,\ldots$. Then,
 $\mathbb{O}=\cup_{j=1}^{\infty}\mathbb{O}^{(j)}$. By definiton of
 $\mathbb{M}_{\mathbb{O}}$, if $\mu\in \mathbb{M}_{\mathbb{O}}$, the
 measure $\mu_j(\cdot)=\mu(\cdot\cap \mathbb{O}^{(r_j)})$ and hence
 $\mu$ is $\sigma$-finite. Part (i) is a property of $\sigma$-finite
 measures; see Appendix A1.6, \cite{Daley:2003tf} for details.  

Since $\mu_j$ is finite, Proposition 9.1.III in \cite{Daley:2008kf}
implies that $\mu_j$ is a counting measure if and only if all the
three conditions in (ii) are satisfied. Moreover, if $\mu_j$ is a
counting measure, all of its atoms must lie in
$\mathbb{O}^{(r_j)}$. Because $\mathbb{O}^{(r_j)}$ are disjoint sets and
$\mu=\sum_{j=1}^{\infty}\mu_j$, the measure $\mu$ is a counting measure if and only if
all the three conditions in (ii) are satisfied. 
\end{proof}

The following theorem is an application of Lemma~\ref{lem:decomposition}.
\begin{theorem}
  \label{thm:close}
$\mathbb{N}_{\mathbb{O}}$ is a closed subset of $\mathbb{M}_{\mathbb{O}}$. 
\end{theorem}
\begin{proof}
It is enough to show that the limit of a sequence in
$\mathbb{N}_{\mathbb{O}}$ is still in $\mathbb{N}_{\mathbb{O}}$. Let
$(N_k)_{k\in \mathbb{N}}$ be a sequence of counting measures and
$N_k\to N$ in $\mathbb{M}_{\mathbb{O}}$. Let $y$ be an arbitrary point
in $\mathbb{O}$. Since $N\in \mathbb{M}_{\mathbb{O}}$, for all but a
countable set of values of $r\in (0,d(y,\mathbb{C}))$, $N(\partial
B_{r}(y))=0$, where $\partial A$ is the boundary of a set $A\in
\mathcal{S}_{\mathbb{O}}$. We can find a decreasing sequence
$(r_j)_{j\in \mathbb{N}}$ such that $\lim_{j\to \infty}r_j=0$ and
$N(\partial B_{r_j}(y))=0$, $j\ge 1$. By Portmanteau theorem (Theorem
2.1, \cite{Lindskog:2014ia}), we have that for $j\ge 1$,
\begin{eqnarray*} 
N_k(B_{r_j}(y))\to N(B_{r_j}(y))\,, \quad k\to\infty\,.
\end{eqnarray*} Since $N_k(B_{r_j}(y))$ are non-negative integers,
$N(B_{r_j}(y))$ are also non-negative integers and thus $N$ is a
counting measure by Lemma~\ref{lem:decomposition}.
\end{proof}

\section{Random measures, point processes and weak convergence}
\label{sec:rand-meas-lapl}
In this section, the properties of random measures and point
processes are studied. We will show that weak convergence is
determined by the finite-dimensional convergence. By applying this
result, it can be shown that weak convergence is equivalent to
convergence of Laplace functionals, which will be frequently used 
in the following sections.
\begin{definition}
\begin{enumerate}[(i)]
  \item A random measure $\xi$ with state space $\mathbb{O}$ is a measurable
    mapping from a probability space $(\Omega,\mathcal{E}, P)$ into
    $(\mathbb{M}_{\mathbb{O}},\mathcal{M}_{\mathbb{O}})$. 
\item\label{item:2} A point process on $\mathbb{O}$ is a measurable mapping
  from a probability space $(\Omega,\mathcal{A}, P)$ into
  $(\mathbb{N}_{\mathbb{O}},\mathcal{N}_{\mathbb{O}})$. 
  \end{enumerate}
\end{definition}
A realization of a random measure $\xi$ has the value $\xi(A,\omega)$
on the borel set $A\in \mathcal{S}_{\mathbb{O}}$. For each fixed $A$,
$\xi_A=\xi(A,\cdot)$ is a function mapping $\Omega$ into
$\mathbb{R}^+=[0,\infty]$. The following proposition provides a convenient way to
examine whether a mapping is a random measure. 
\begin{theorem}\label{thm:randommeasure}
Let $\xi$ be a mapping from a probability space $(\Omega,
\mathcal{E},P)$ into $\mathbb{M}_{\mathbb{O}}$ and $\mathcal{D}$ be a
$\pi$-system of Borel sets bounded away from $\mathbb{C}$, which generates
$\mathcal{S}_{\mathbb{O}}$. Then $\xi$ is a random measure if and only
if $\xi_A$ is a random variable for each $A\in
\mathcal{D}$. Similarly, $N$ is a point process if and only if $N(A)$
is a random variable for each $A\in \mathcal{D}$.
\end{theorem}
\begin{proof}
Let $\mathcal{U}$ be the $\sigma$-algebra of subsets of
$\mathbb{M}_{\mathbb{O}}$ whose inverse images under $\xi$ are events, and
let $\Phi_A$ denote the mapping taking a measure $\mu\in
\mathbb{M}_{\mathbb{O}}$ into $\mu(A)$. Because
$\xi_A(\omega)=\xi(A,\omega)=\Phi_A(\xi(\cdot,\omega))$, 
\begin{eqnarray*}
  \xi^{-1}(\Phi_A^{-1}(B))=(\xi_A)^{-1}(B)\,, \quad B\in \mathcal{B}(\mathbb{R}^+)\,.
\end{eqnarray*}
If $\xi_A$ is a random variable, $(\xi_A)^{-1}(B)\in \mathcal{E}$ and
we have $\Phi_A^{-1}(B)\in \mathcal{U}$ by
definition. Proposition~\ref{prop:algebra} yields that
$\mathcal{M}_{\mathbb{O}}\subset \mathcal{U}$ and thus $\xi$ is a
random measure. Conversely, if $\xi$ is a 
random measure, $\Phi_A^{-1}(B)\in \mathcal{M}_{\mathbb{O}}$ for
$B\in \mathcal{B}(\mathbb{R}^+)$ and
hence, $\xi^{-1}(\Phi^{-1}_A(B))\in \mathcal{E}$. This shows that
$\xi_A$ is a random variable. 

It is similar to show that $N$ is a point process if and only if
$N(A)$ is a random variable for each $A\in \mathcal{D}$. 
\end{proof}
As a simple application of Proposition~\ref{prop:algebra} and
Theorem~\ref{thm:randommeasure}, it is easy to prove the following
corollary and hence the proof is omitted. 
\begin{corollary}
The sufficient and necessary condition for $\xi$ to be a random
measure is that $\xi(A)$ is a ramdom variable for each $A\in
\mathcal{S}_{\mathbb{O}}$ bounded away from $\mathbb{C}$. Similarly,
the sufficient and necessary condition for $N$ to be a point process
is that $N(A)$ is a random variable for each $A\in
\mathcal{S}_{\mathbb{O}}$ bounded away from $\mathbb{C}$.
\end{corollary}
\subsection{Finite-dimensional distributions}
The finite-dimensional, or \emph{fidi} for short, distributions of a
random measure $\xi$ are the joint distributions for all finite
families of the random variables $\xi(A_1), \ldots, \xi(A_k)$, where
$A_1, \ldots, A_k$ are Borel sets bounded away from $\mathbb{C}$, that is,
the family of proper distribution functions
\begin{eqnarray}\label{eq:fidi}
  F_k(A_1,\ldots,A_k;x_1,\ldots,x_k)=P(\xi(A_i)\le x_i,
  i=1,\ldots,k)\,.
\end{eqnarray}
\begin{theorem}\label{thm:fidi}
The distribution of a random measure is completely determined by the
fidi distributions \eqref{eq:fidi} for all finite families $\{
A_1,\ldots,A_k\}$ of disjoint sets from a $\pi$-system of Borel sets
bounded away from $\mathbb{C}$ generating $\mathcal{S}_{\mathbb{O}}$.  
\end{theorem}
Note that for a $\pi$-system $\mathcal{A}$, if two probability
measures $P_1$ and $P_2$ agree on $\mathcal{A}$, then $P_1$ and $P_2$
agree on $\sigma(\mathcal{A})$; see Theorem 3.3,
\cite{BIL95}. The Borel sets bounded away from
$\mathbb{C}$ forms a $\pi$-system and hence Theorem~\ref{thm:fidi}
holds by Proposition~\ref{prop:algebra}. The proof is completed. 

Similarly, the fidi distributions of the point process $N$ are the
joint distributions for Borel sets
bounded away from $\mathbb{C}$, $\{A_1,\ldots,A_2\}$ 
and nonnegative integers $n_1,n_2,\ldots$, is defined by
\begin{eqnarray*}
  P_k(A_1,\ldots,A_k;n_1,\ldots,n_k)=P(N(A_i)=n_i\,, i=1,\ldots,k)\,.
\end{eqnarray*}
According to Theorem~\ref{thm:fidi}, the fidi distributions determine
$N$.  

\subsection{Laplace functionals}
\label{sec:laplace-functionals}
Let $BM(\mathbb{S})$ be a class of non-negative bounded measurable
functions $f$ for which there exists $r>0$ such that $f$ vanishes on
$\mathbb{C}^r$. Let
$\xi:(\Omega,\mathcal{E}, P)\to (\mathbb{M}_{\mathbb{O}},
\mathcal{M}_{\mathbb{O}})$ be a random measure and $N$ be a point
process. The Laplace functional is defined for a random measure $\xi$
and each $f\in BM(S)$ by
\begin{equation}
\label{eq:laplace-rm}
  L_{\xi}[f]=E[\exp(-\xi(f)]=\int_{\mathbb{M}_{\mathbb{O}}}\exp\Big(-\int_{\mathbb{O}} f(x)\,\xi(dx)\Big)\,P(d\xi)\,,
\end{equation}
where $\xi_f=\int_{\mathbb{O}}f(x)\,\xi(dx)$. Similarly,
we can define $L_N[f]$ for the point process $N$ and each $f\in BM(\mathbb{S})$. 
\begin{theorem}\label{thm:laplace}
The Laplace functions $\{L_{\xi}[f]: f\in BM(\mathbb{S})\}$ uniquely determine
the distribution of a random measure $\xi$. Similarly, the Laplace
functions $\{L_N[f]: f\in BM(\mathbb{S})\}$ uniquely determine the distribution
of a point process $N$.  
\end{theorem}
\begin{proof}
 For $k\ge 1$ and Borel sets $A_1,\ldots,A_k\in \mathcal{S}_{\mathbb{O}}$
 bounded away from $\mathbb{C}$ and $\lambda_i>0$, $i=1\ldots, k$, the function
 $f:\mathbb{O}\to [0,\infty)$ is given by
\begin{equation*}
f(x)= \sum_{i=1}^k\lambda_i\mathbf{1}_{A_i}(x)\,, \quad x\in \mathbb{O}\,,
\end{equation*}
where $\mathbf{1}_A$ is the indicator function for $A\in
\mathcal{S}_{\mathbb{O}}$. Then for each realization $\omega\in \Omega$, 
\begin{eqnarray*}
  \xi(\omega,f)= \int_{\mathbb{S}}f(x)\xi(\omega,dx)=\sum_{i=1}^k\lambda_i\xi(\omega,A_i)\,,
\end{eqnarray*}
and 
\begin{eqnarray*}
 L_{\xi}[f]= E\exp \Big(-\sum_{i=1}^k\lambda_i\xi(A_i) \Big)\,,
\end{eqnarray*}
which is the joint Laplace transform of the random vector
$(\xi(A_i))_{i=1,\ldots,k}$. The uniqueness of theorem for Laplace
transform for random vectors yields that $L_{\xi}$ uniquely determines
the law of $(\xi(A_i))_{i=1,\ldots,k}$ and
Theorem~\ref{thm:fidi} completes the proof.
\end{proof}
The following proposition shows the convergence of Laplace
functionals. 
\begin{proposition}\label{prop:lp-convergence}
  Let $\xi$ be a random measure. For a sequence of functions
  $(f_n)_{n\in \mathbb{N}}$ with $f_n\in BM(\mathbb{S})$ and a
  function $f\in BM(\mathbb{S})$, the convergence $L_{\xi}[f_n]\to
  L_{\xi}[f]$ as $\sup_{x\in \mathbb{O}}|f_n(x)-f(x)|\to
  0$ if one of three conditions holds: (i) $\xi(\mathbb{O})<\infty$;
  (ii) the pointwise convergence $f_n\to f$ is monotonic;
(iii) there exists $r>0$ such that for each $n\ge 1$, $f_n$ vanishes
on $\mathbb{C}^r$.  
\end{proposition}
\begin{proof}
If condition (i) holds, 
\begin{eqnarray*}
  |L_{\xi}[f_n]-L_{\xi}[f]|\le \xi(\mathbb{O})\sup_{x\in \mathbb{O}}|f_n(x)-f(x)|\to 0\,.
\end{eqnarray*}  
If condition (iii) holds, since $\xi(\mathbb{S}\setminus \mathbb{C}^r)<\infty$, 
\begin{eqnarray*}
|L_{\xi}[f_n]-L_{\xi}[f]|\le \xi(\mathbb{S}\setminus \mathbb{C}^r)\sup_{x\in \mathbb{O}}|f_n(x)-f(x)|\to 0\,.
\end{eqnarray*}
Suppose that condition (ii) holds. If $f_n(x)\downarrow f(x)$ for each
$x\in \mathbb{O}$, this implies that there exists $r>0$ that
$f_n$ vanishes on $\mathbb{C}^r$ because $f_1\in BM(\mathbb{S})$ and condition (iii)
is satisfied. If $f_n(x)\uparrow f(x)$ for each $x\in
\mathbb{O}$, dominated convergence theorem ensures that $L_{\xi}[f_n]\to L_{\xi}[f]$. 
\end{proof}
\subsection{Weak convergence of random measures}
\label{sec:weak-conv-rand}
Weak convergence is characterized by weak convergence of fidi
distributions. In connection of 
this idea, we are interested in the \emph{stochastic continuity sets},
that is, Borel sets $A\in \mathcal{S}_{\mathbb{O}}$ bounded away from
$\mathbb{C}$ satisfying the condition that $P(\xi(\partial A)>0)=0$ for a probability measure $P$. Let $S_P$ be the collection of such sets. It is trivial to show
that for $A,B\in S_P$, $A\cup B\in S_P$ and $A\cap B\in S_P$. Hence,
$S_P$ is a $\pi$-system. The following lemma implies that $S_P$
generates $\mathcal{S}_{\mathbb{O}}$.
\begin{lemma}\label{lem:fidifinite}
  Let $P$ be a probability measure on $\mathcal{M}_{\mathbb{O}}$ and
  $S_P$ be the class of stochastic continuity sets for
  $P$. Then for all but a countable set of values of $r>0$,
  $\mathbb{C}^r\in S_P$ and given $x\in \mathbb{O}$, for all but a
  countable set of values of $r\in (0,d(x,\mathbb{C}))$, $B_{r}(x)\in
  S_P$. 
\end{lemma}
\begin{proof}
  It is sufficient to show that for each finite positive
  $\varepsilon_1$, $\varepsilon_2$ and $r_0$, there are only finite
  numbers of $r>r_0$ satisfying 
\begin{equation}
\label{eq:fidifinite1}
P\Big(\xi\Big(\partial(\mathbb{C}^r)>\varepsilon_1 \Big)\Big)>\varepsilon_2\,.
\end{equation}
Suppose the contrary and there exists positive numbers $\varepsilon_1$,
$\varepsilon_2$ and $r_0$ such that there is a countably infinite set $\{r_i,
i\ge 1\}$ with $r_i >r_0$ satisfying 
\[P\big(A_i \big)>\varepsilon_2\,, \quad i\ge 1\,,\]
where $A_i=\{\xi\in \mathbb{M}_{\mathbb{O}}:\xi(\partial
\mathbb{C}^{r_i})>\varepsilon_1\}$. Therefore,
\begin{align*}
\varepsilon_{2}\le \limsup_{i\to \infty}P(A_i)\le P(\limsup_{i\to
  \infty}A_i)\le P(\xi(\mathbb{S}\setminus \mathbb{C}^{r_0})=\infty).
\end{align*}
This contradicts to the assumption that for any $\xi\in
\mathbb{M}_{\mathbb{O}}$ and $r>0$, $\xi(\mathbb{S}\setminus
\mathbb{C}^r)<\infty$. Thus, \eqref{eq:fidifinite1} holds for a finite
number of $r>r_0$. 

Applying a similar arguments, we can easily show that given
$x\in \mathbb{O}$, for all but a countable set of values of $r\in
(0,d(x,\mathbb{C}))$, $B_{r}(x)\in S_P$.
\end{proof}

We say that a sequence of random measures $(\xi_n)_{n\ge 1}$ converges
in the sense of fidi distributions if for every finite family
$\{A_1,\ldots,A_k\}$ with $A_i\in S_P$, the joint distributions of
$(\xi_n(A_1),\ldots,\xi_n(A_k))$ converge weakly in
$\mathcal{B}(\mathbb{R}^k)$ to the joint distribution of
$(\xi(A_1),\ldots,\xi(A_k))$. Let $(P_n)_{n\ge 1}$ be a sequence of
probability measures and assume that $P_n\to P$ weakly, or
$P_n\overset{w}{\to} P$. Let $A$ be a
stochastic continuity set for $P$ and denote the mapping
$\Phi_A:\xi\mapsto \xi(A)$. According the definition of the stochastic
continuity set, $P(D)=0$ where $D$ is the set of discontinuity points
for $\Phi_A$. By continuous mapping theorem (see Proposition A2.3.V in
\cite{Daley:2003tf}),  
\begin{align*}
P_n(\Phi_A^{-1})\to P(\Phi_A^{-1})\,.
\end{align*}
Similarly, for any finite family $\{A_1,\ldots,A_k\}$ with $A_i\in
S_P$, let $\Phi_k:\xi\mapsto (\xi(A_1),\ldots,\xi(A_k))$ and
$P_n(\Phi_k^{-1})\to P(\Phi_k^{-1})$ if $P_n\overset{w}{\to} P$. This
leads to the following lemma. 
\begin{lemma}\label{lem:wc}
Weak convergence implies weak convergence of the finite-dimensional
convergence. 
\end{lemma}

To show that weak convergence of the finite-dimensional convergence
implies weak convergence, we need to show that the converging sequence
of probability measures are \emph{uniformly tight}. A family of
probability measures $(P_t)_{t\in \mathcal{T}}$ with an index set
$\mathcal{T}$ on $\mathcal{M}_{\mathbb{O}}$ is \emph{uniformly tight}
if for each $\varepsilon>0$, there exists a compact set $K\in
\mathcal{M}_{\mathbb{O}}$ such that $P_t(K)>1-\varepsilon$ for all
$t\in \mathcal{T}$.
\begin{lemma}\label{lem:tight}
For a family of probability measures $(P_t)_{t\in \mathcal{T}}$ on
$\mathcal{M}_{\mathbb{O}}$ to uniformly tight, it is necessary and
sufficient that there exists a sequence $(r_i)$ with $r_i\downarrow 0$
such that for each $i$ and any $\varepsilon,\varepsilon^{\prime}>0$, there exist
real numbers $M_i>0$ and compact set $K_i\subset \mathbb{S}\setminus
\mathbb{C}^{r_i}$ such that, uniformly for $t\in \mathcal{T}$,
\begin{align}
\label{eq:tight1}
P_t\big(\xi(\mathbb{S}\setminus
  \mathbb{C}^{r_i})>M_i\big)&<\varepsilon\,,\\ \label{eq:tight2}
P_t\big(\xi(\mathbb{S}\setminus (\mathbb{C}^{r_i}\cup
  K_i))<\varepsilon^{\prime}\big)&<\varepsilon\,.
\end{align}   
\end{lemma}
The proof of Lemma~\ref{lem:tight} is in Section~\ref{sec:proof}. The next theorem is the main result in this section and its proof is
in Section~\ref{sec:proof}. 
\begin{theorem}\label{thm:fidiconvergence}
Let $(P_n)_{n\ge 1}$ and $P$ be distributions on
$\mathcal{M}_{\mathbb{O}}$. Then, $P_n\overset{w}{\to} P$ if and only if the
fidi distribution of $P_n$ converge weakly to those of $P$.  
\end{theorem}
Two equivalent conditions for weak convergence can be derived as the
corollaries to Theorem~\ref{thm:fidiconvergence}. 
\begin{corollary}\label{cor:lp}
 The two following conditions is equivalent to the weak
  convergence $P_n\to P$ with $f\in \mathcal{C}_{\mathbb{O}}$:
  \begin{enumerate}[(i)]
  \item The distribution of
    $\int_{\mathbb{O}}f\,d\xi$ under $P_n$ converges weakly to its
    distribution under $P$.
\item The Laplace functionals
  $L_n[f]=E_{P_n}\Big(\exp(-\int_{\mathbb{O}}f(x)\xi(dx)) \Big)$
  converge pointwise to the limit functional
  $L[f]=E_{P}\Big(\exp(-\int_{\mathbb{O}}f(x)\xi(dx)) \Big)$. 
 \end{enumerate}
\end{corollary}
\begin{proof}
  Condition (ii) is equivalent to (i) by well-known results on Laplace
  transform. 

Think of the simple functions of the form $f=\sum_{i=1}^kc_i\bm{1}_{A_i}$,
where $k$ is a finite positive integer, $\sum_i|c_i|<\infty$ and
$(A_i)_{i\ge 1}$ are a family of Borel sets with $A_i\in
S_P$. Convergence of distributions of the integrals
$\int_{\mathbb{O}}f\,d\xi$ is equivalent to the finite-dimensional
convergence for every finite $k$. Following a classical arguments, we
can find $h^+_l,h^-_l\in \mathcal{C}_{\mathbb{O}}$ satisfying that
$0<h^-_l(x)\uparrow f(x)$ and $h^+_l(x)\downarrow f(x)$ 
holds uniformly for every $x\in \mathbb{S}$ as $l\to\infty$. It
follows that condition (i) and Proposition~\ref{prop:lp-convergence}
implies the finite-dimensional convergence 
and hence weak convergence.  
\end{proof}

\section{Poisson processes}
\label{sec:poisson-processes}
Poisson processes play a vital role in the applications of the point
process and we start with the definition of a Poisson process. 

\begin{definition}\label{def:prm}
Given a random measure $\mu\in \mathbb{M}_{\mathbb{O}}$, a point process $N$ is called a {\em
  Poisson process} or {\em Poisson random measure} (PRM) with mean
measure $\mu$ if $N$ satisfies
\begin{enumerate}[(i)]
\item For any $A\in \mathcal{S}_{\mathbb{O}}$ and any non-negative integer
  $k$, 
\begin{eqnarray*}
  P(N(A)=k)=
  \begin{cases}
    \exp(\mu(A))(\mu(A))^{k}/k!\,, & \mu(A)<\infty\,, \\
0\,, & \mu(A)=\infty\,.
  \end{cases}
\end{eqnarray*}
\item\label{item:5} For any $k\ge 1$, if $A_1, \ldots, A_k$ are
  mutually disjoint Borels sets bounded away from $\mathbb{C}$, then $N(A_i)$,
  $i=1,\ldots, k$ are independent random variables. 
\end{enumerate}
\end{definition}
For short, we write {\em Poisson process with mean measure $\mu$} as
{\bf PRM($\bm{\mu}$)}. 

\begin{proposition}\label{prop:prm}
  PRM($\mu$) exists and its law is uniquely determined by conditions
  (i) and (ii) in Definition~\ref{def:prm}. Moreover, the Laplace
  functional of PRM($\mu$) is given by
\begin{equation}\label{eq:lpprm}
  L_N[f]=\exp\Big(- \int_{\mathbb{O}}(1-e^{-f(x)})\mu(dx) \Big)\,, \quad f\in BM(\mathbb{S})\,,
\end{equation}
and conversely a point process with Laplace functional of the form
\eqref{eq:lpprm} must be PRM($\mu$). 
\end{proposition}
\begin{proof}[Sketch of the proof.]
  To prove that the Laplace functional of
  PRM($\mu$) is given by \eqref{eq:lpprm}, one can choose
  $f=c\bm{1}_A$ for $c\ge 0$ and $A\in \mathcal{S}_{\mathbb{O}}$
  bounded away from $\mathbb{C}$. Following the lines in the proof of
  Proposition 3.6, \cite{RES87}, the Laplace functional $L_N[f]$ has
  the form \eqref{eq:lpprm}. Let 
\begin{align}
\label{eq:simplefunction}
f=\sum_{i=1}^kc_i\bm{1}_{A_i}\,,
\end{align}
  where $k>0$, $c_i\ge 0$ and $A_1,\ldots,A_k$ are disjoint sets
  bounded away from $\mathbb{C}$. Similarly, it can be shown that the
  Laplace functional $L_N[f]$ has the form \eqref{eq:lpprm}. Then for
  any $f\in BM(\mathbb{S})$, there exists simple functions $f_n$ of
  the form \eqref{eq:simplefunction} such that $f_n\uparrow f$ with
  $\sup_{x\in \mathbb{O}}|f_n(x)-f(x)|\to 0$ as $n\to \infty$. By
  Proposition 3.6, we have that the Laplace functional $L_N[f]$ has
  the form \eqref{eq:lpprm}. Conversely, it is easy to prove following
  the lines of the lines in the proof of Proposition 3.6,
  \cite{RES87}.  

The proof of the existence of PRM($\mu$) is through construction. we
will use the same trick in the proof of Lemma~\ref{lem:decomposition}
to divide $\mathbb{O}$ into countable disjoint subspaces
$\mathbb{O}^{(r_j)}$, $j=1,2,\ldots$. Then let $\mu_j(\cdot)=\mu(\cdot
\cap \mathbb{O}^{(r_j)})$ for $\mu\in \mathbb{M}_{\mathbb{O}}$. Using
the arguments in the proof of Proposition 3.6, \cite{RES87}, it is
easy to construct PRM($\mu_i$) for $i\ge 1$, named $N_i$. Let
$N=\sum_{i=1}^{\infty}N_i$. For $f\in BM(\mathbb{S})$,
\begin{eqnarray*}
L_N[f] &= & E\exp(-\sum_{i=1}^{\infty}N_i(f)) =\lim_{n\to
            \infty}E\exp\Big(-\sum_{i=1}^nN_i(f)\Big)\\ 
 & =& \lim_{n\to \infty}\prod_{i=1}^nE\exp(-N_i(f))\\ 
 & =& \lim_{n\to \infty}\prod_{i=1}^n E\exp \Big(
      -\int_{\mathbb{O}}(1-e^{-f(x)})\mu_i(dx)\Big)\\ 
 & =& \exp \Big(-\int_{\mathbb{O}}(1-e^{-f(x)})\sum_{i=1}^{\infty}\mu_i(dx)
      \Big)\\ 
 & =& \exp \Big(-\int_{\mathbb{O}}(1-e^{-f(x)})\mu(dx) \Big)\,.
\end{eqnarray*}
This shows that PRM($\mu$) exists. 
\end{proof}
A new Poisson process can be constructed by mapping points of a
Poisson process. Recall that the Dirac measure
$\delta_x$ for $x\in \mathbb{S}$ is defined by 
\begin{eqnarray*}
  \delta_x(A)=
  \begin{cases}
    1\,, & x\in A\,,\\
0\,, & x\notin A\,,
  \end{cases}
\quad A\in \mathcal{S}\,.
\end{eqnarray*}
\begin{proposition}\label{prop:trans1}
Assume that the complete and separable spaces $(\mathbb{S}_1,d_1)$ and
$(\mathbb{S}_2,d_2)$ with closed cones $\mathbb{C}_1\subset
\mathbb{S}_1$ and $\mathbb{C}_2\subset \mathbb{S}_2$ are equipped with
scalar multiplication. Let $\mathbb{O}_1=\mathbb{S}_1\setminus
\mathbb{C}_1$ and $\mathbb{O}_2=\mathbb{S}_2\setminus
\mathbb{C}_2$. Denote $\mathcal{S}_{\mathbb{O}_i}$, $i=1,2$ as the
corresponding $\sigma$-algebra of $\mathbb{O}_i$, $i=1,2$,
correspondingly. Let $T:(\mathbb{S}_1,\mathcal{S}_{\mathbb{O}_1})\to
(\mathbb{S}_2,\mathcal{S}_{\mathbb{O}_2})$ be a measurable mapping
satisfying the condition that for every $\varepsilon>0$, 
\begin{equation}\label{eq:away}
  \inf\{d_1(x,\mathbb{C}_1):d_2(Tx,\mathbb{C}_2)>\varepsilon\}>0\,.
\end{equation}
Then if $N$ is PRM($\mu$) on $\mathbb{O}_1$, then
\begin{eqnarray*}
\widetilde{N}_1=N\circ T^{-1}
\end{eqnarray*}  
is PRM($\widetilde{\mu}$) on $\mathbb{O}_2$ with
$\widetilde{\mu}=\mu\circ T^{-1}$. If we have a representation
\begin{eqnarray*}
  N=\sum_i \delta_{X_i}\,,
\end{eqnarray*}
then $\widetilde{N}=N\circ T^{-1}=\sum_{i}\delta_{TX_i}$. 
\end{proposition}
\begin{proof}
  Let $f\in BM(\mathbb{S}_2)$ and there exist $r>0$ such that $f$
  vanishes on $\mathbb{C}_2^r=\{x\in
  \mathbb{O}_2:d(x,\mathbb{C}_2)<r\}$. The condition \eqref{eq:away}
  implies that for $r>0$, there exists $\widetilde{r}>0$ such that
  $T^{-1}(x_2)\in \mathbb{S}_1\setminus \mathbb{C}_1^{\widetilde{r}}$ for
  $x\in \mathbb{S}_2\setminus \mathbb{C}_2^r$. Therefore, $f\circ T
  \in BM(\mathbb{S}_1)$ and $\mu\circ T^{-1}\in
  \mathbb{M}_{\mathbb{O}_2}$. An application of
  Proposition~\ref{prop:prm} yields that
\begin{eqnarray*}
L_{\widetilde{N}}[f] &=& E\exp(-\widetilde{N}(f))=E\exp
                         \Big(-\int_{\mathbb{O}_2}f(x_2)\,N\circ
                         T^{-1}(\omega,dx_2) \Big)\\
 & =& E\exp \Big(-\int_{\mathbb{O}_1}f(Tx_1)\,N(\omega,dx_1) \Big)\\ 
 & =& \exp \Big(-\int_{\mathbb{O}_1}(1-e^{-f\circ T})\,d\mu\Big)\\ 
 & =& \exp \Big(-\int_{\mathbb{O}_2}(1-e^{-f(x)})\,\mu\circ T^{-1}(dx)
      \Big)\,,
\end{eqnarray*}
which is the Laplace functional of PRM($\mu\circ T^{-1}$) on
$\mathbb{S}_2$. 
\end{proof}

We next construct a new PRM living in a higher dimensional space from a
given PRM. Let $(\mathbb{S},d_1)$ with a closed cone $\mathbb{C}$ be a
complete and separable space and let $(\mathbb{K},d_2)$ be a
complete and separable space. Due
to the separability of $\mathbb{K}$, there exists a countable
dense $D\subset \mathbb{K}$ and $r_0>0$ such that the open
neighborhoods $\{y\in \mathbb{K}:d(x,y)<r_0\}$ for each
$x\in D$ covers $\mathbb{K}$. Denote $\mathcal{K}$ as the Borel
$\sigma$-algebra on $\mathbb{K}$. We define a transition function
$G:\mathbb{K}\times \mathcal{O}\to [0,1]$ from $\mathbb{O}$ to
$\mathbb{K}$: $G(F,\cdot)$ is $\mathcal{S}$-measurable for $F\in
\mathcal{K}$ and $G(\cdot,x)$ is a probability measure on
$\mathcal{K}$ for each $x\in \mathbb{O}$. Assume that the sequence of
random elements $(K_n)_{n\ge 1}$ taking values in $\mathbb{K}$ is
conditionally independent given the sequence $(X_n)_{n\ge 1}$ taking
values in $\mathbb{O}$, that is,
\begin{equation}
\label{eq:condindep}
P\Big(K_i\in F\mid (X_n)_{n\ge 1}, (K_j)_{j\neq n}\Big)=G(X_i,F)\,, \quad i\ge
1\,, F\in \mathcal{K}\,.
\end{equation}
Consider the space $\widetilde{\mathbb{S}}=(\mathbb{K},\mathbb{S})$
assigned with the Euclidean distance
\begin{align}
\label{eq:ddis}
\widetilde{d}((x_1,y_1),(x_2,y_2))=\sqrt{(d_1(x_1,x_2))^2+(d_2(y_1,y_2))^2}\,,
\end{align}
where $(x_i,y_i)\in \mathbb{K}\times \mathbb{S}$, $i=1,2$. Let
$\mathbb{\widetilde{C}}=\mathbb{K}\times \mathbb{C}$ and
$\mathbb{\widetilde{O}}=\mathbb{\widetilde{S}}\setminus
\mathbb{\widetilde{C}}$. The set $\mathbb{\widetilde{C}}$ is
closed. Now we define a scalar multiplication on
$\mathbb{\widetilde{S}}$ as a map $[0,\infty)\times
\mathbb{\widetilde{S}}\to \mathbb{\widetilde{S}}: (\lambda, (x,y))\to
(\lambda x,y)$. It is trivial to show that the conditions (i), (iii) and (iii) in
Definition~\ref{def:scalar} are satisfied. Since
$\widetilde{d}((x,y),\mathbb{\widetilde{C}})=d_1(x,\mathbb{C})$ for
$(x,y)\in \mathbb{\widetilde{S}}$, the condition (iv) in
Definition~\ref{def:scalar} is also satisfied. Moreover,
$\mathbb{\widetilde{C}}$ is also a cone. According to the arguments
above, the space
$(\mathbb{M}_{\mathbb{\widetilde{O}}},\mathcal{M}_{\mathbb{\widetilde{O}}})$
is also well defined. The set
\[\mathbb{\widetilde{C}}^r=\{(x,y):\widetilde{d}((x,y),\mathbb{\widetilde{C}})<r\}=\{(x,y):d_1(x,\mathbb{C})<r\}\,,\]
is then an open set. Let $BM(\mathbb{\widetilde{S}})$ be a class of
the bounded and measurable functions $f:\mathbb{\widetilde{S}}\to
[0,\infty)$ vanishing on $\mathbb{\widetilde{C}}^{r}$ for some $r>0$.
Denote $\mathcal{\widetilde{S}}_{\mathbb{\widetilde{O}}}$ as the
Borel $\sigma$-algebra of $\mathbb{\widetilde{O}}$ generated by a
$\pi$-system consisting of the sets of the form $A\times F$ with $F\in
\mathcal{K}$ and $A\in \mathcal{B}(\mathbb{C}^r)$ for $r>0$.

Similar to Lemma 3.9 and 3.10 in \cite{RES87}, we have the following
lemma. 
\begin{lemma}\label{lem:310}
Assume that $f:\mathbb{\widetilde{O}}\to [0,1]$ belongs to $BM(\mathbb{\widetilde{S}})$. Then for $i\ge 1$, 
\begin{equation}
\label{eq:prod1}
E \Big(f(K_i,X_i)\mid (X_n)_{n\ge
  1}\Big)=\int_{\mathbb{K}}f(y,X_i)\,G(dy,X_i)\,, \quad a.s.\,,
\end{equation}
and
\begin{equation}
\label{eq:prod2}
E \Big(\prod_{i=1}^{\infty}f(K_i,X_i)\mid (X_n)_{n\ge 1}
\Big)=\prod_{i=1}^{\infty}E\Big(f(K_i,X_i)\mid (X_n)_{n\ge 1}\Big)\,,
\quad a.s.\,.
\end{equation}
\end{lemma}
The following proposition is proved by using Laplace functional and
the proof is similar to the proof of Proposition 3.8 in
\cite{RES87}. Therefore, the proof is omitted. 
\begin{proposition}
  Let $\mathbb{\widetilde{S}}=(\mathbb{S},\mathbb{J})$ with $K$ as the
  transition function be the space defined before
  Lemma~\ref{lem:310}. Suppose that  
\begin{equation*}
N= \sum_{i}\delta_{X_i}
\end{equation*}
is PRM($\mu$) on $(\mathbb{S},d_1)$. Then 
\begin{equation*}
N^{*}=\sum_{i}\delta_{(X_i,K_i)}
\end{equation*}
is PRM on $\mathbb{\widetilde{S}}$ with mean measure
\begin{equation*}
\mu^{*}(dx, dy)=\mu(dx)G(x,dy)\,.
\end{equation*}
\end{proposition}

\section{Regular variation}
\label{sec:regular-variation}
For $\tau\in \mathbb{R}$, let $\mathcal{R}_{\tau}$ denote the class of regularly
varying functions at infinity with index $\tau$, i.e., positive,
measurable functions $g$ defined in a neighbourhood of infinity such
that $\lim_{u\to \infty}g(\lambda u)/g(u)=\lambda^{\tau}$ for every
$\lambda\in (0,\infty)$. 
\begin{definition}\label{def:rv}
A sequence of measures $(\nu_n)_{n\ge 1}$ in $\mathbb{M}_{\mathbb{O}}$
is \emph{regularly varying} if there exists a nonzero $\mu\in
\mathbb{M}_{\mathbb{O}}$ and a regularly varying function $b$ with
index $-\alpha^{-1}$, $\alpha>0$ such that $t\nu\big(b(t)\cdot\big)\to
\mu$ in $\mathbb{M}_{\mathbb{O}}$ as $t\to \infty$. 
\end{definition}
The limiting measure $\mu$ in Definition~\ref{def:rv} has the
homogeneity property
\begin{align}\label{eq:homo}
\mu(\lambda A)=\lambda^{-\alpha}\mu(A)\,, \quad
  \alpha>0\,,\lambda>0\,, A\in \mathcal{S}_{\mathbb{O}}\,. 
\end{align}
This property is proved by Theorem 3.1 in
\cite{Lindskog:2014ia}. Similarly, a random element $X$ in
$\mathbb{S}$ is regularly varying with index $-\alpha$ if and only if
there exists a 
nonzero $\mu\in \mathbb{M}_{\mathbb{O}}$ and a regularly varying
function $b$ with index $-\alpha^{-1}$, $\alpha>0$ such that 
\begin{align*}
tP(X\in b(t)\cdot)\to \mu(\cdot)\,, \quad t\to \infty\,.
\end{align*}
The measure $\mu$ on the right-hand side must have the property
\eqref{eq:homo}. 

Consider an iid sequence $(X_n)_{n\ge 1}$, where $X_{n}$ is regularly
varying with index $-\alpha<0$ for each $n$. Define the point process
\begin{align}\label{eq:ppp}
N_n=\sum_{i=1}^n\delta_{(i/n,X_i/b(n))}\,.
\end{align}
It is well-known that regular variation of each $X_i$ is equivalent to
the convergence of the point process \eqref{eq:ppp} towards a suitable
PRM; see Proposition 3.21, \cite{RES87}. The following theorem is a
reproduction of this result under
$\mathbb{M}_{\mathbb{O}}$-convergence. 
\begin{theorem}\label{thm:iidmain}
  Let $(X_n)_{n\ge 1}$ be an iid copies of a random element $X$ taking
  values in $\mathbb{S}$ and a measure $\mu\in
  \mathbb{M}_{\mathbb{O}}$. Suppose that there exists an increasing
  regularly varying function $b$ with index $-\alpha^{-1}$,
  $\alpha>0$. A point process $N_n$ is defined as \eqref{eq:ppp} and
  the process $N$ is PRM on $[0,\infty)\times \mathbb{O}$ with mean
  measure $dt\times d\mu$. Then, $N_n\overset{w}{\to} N$ if and only
  if 
\begin{align}
\label{eq:mainiid}
tP(X/b(t)\in \cdot)\to \mu\,, \quad t\to \infty\,, \text{ in }
  \mathbb{M}_{\mathbb{O}}\,. 
\end{align}
\end{theorem}
The proof is simply an application of Laplace functionals. By using
Corollary~\ref{cor:lp}, it is easy to prove following the line of the
proof of Proposition 3.21, \cite{RES87}.
\section{Discussions}
\label{sec:discussions}
The choice of the cone $\mathbb{C}$ is vital to the definition of the
space $\mathbb{M}_{\mathbb{O}}$. This topic is already covered by
various literatures in different contexts; see 
\cite{Janssen:2016ki,Kulik:2015cz, Das:2016dd,Das:2016to}, to name a few.  

It is proven in \cite{Lindskog:2014ia} that
$\mathbb{M}_{\mathbb{O}}$-convergence implies vague convergence and
$\mathbb{M}_{\mathbb{O}}\subset \mathbb{M}_+(\mathbb{S}\setminus
\mathbb{C})$. The proof is based on the fact that compact sets in
$\mathcal{S}_{\mathbb{O}}$ are bounded away from $\mathbb{C}$. The
application of vague convergence requires that the space $\mathbb{S}$
is locally compact. For this reason, it is tempting to replace vague
convergence with the $w^{\#}$-convergence in extreme value theory. But
the measures in $\mathbb{M}_{\mathbb{O}}$ does not belong to
$\mathbb{M}_{\#}$, which are the measures frequently used in extreme
value theory. One toy example is to take $\mathbb{S}=\mathbb{R}$ and
$0_S=0$. Find a measure $\mu\in
\mathbb{M}_{\mathbb{O}}$ such that $\mu(\{x:|x|>1/r\})=r$ for all
$r>0$. Note that the set 
$B_r(0)\setminus\{0\}\in
\mathcal{S}_{\mathbb{O}}$ for any fixed $r$ is bounded and $\mu(B_r(0)\setminus
\{0\})>M$ for any $M>0$. We have $\mu\notin
\mathbb{M}_{\#}(\mathbb{O})$ and hence $\mathbb{M}_{\mathbb{O}}$-convergence
does not imply $w^{\#}$-convergence.
\section{Proofs}
\label{sec:proof}
\begin{proof}[Proof of Proposition~\ref{prop:algebra}]

The proof is a modification of Proposition A2.5.I and Theorem A2.6.III,
\cite{Daley:2003tf} under weaker conditions. 

Notice that $\mathbb{S}\setminus \mathbb{C}^{r} $ is close and hence complete and
separable. If $\mu\in \mathbb{M}_{\mathbb{O}}$, the measure
$\mu^{(r)}$ is in $\mathbb{M}_b(\mathbb{S}\setminus
\mathbb{C}^r)$. The space $(\mathbb{M}_b(\mathbb{S}\setminus
\mathbb{C}^r),p_r)$ is complete and separable, where $p_r$ is the
Prohorov metric. We firstly show that for any $\varepsilon>0$, there
exist $\delta_1,\delta_2>0$ and $r>0$ such that for $\mu\in
\mathbb{M}_{\mathbb{O}}$,
\begin{align}
\label{eq:11} \{\upsilon\in
  \mathbb{M}_{\mathbb{O}}:d_{\mathbb{M}_{\mathbb{O}}}(\mu,\upsilon)<\varepsilon\}&\supset
  \{\upsilon\in
  \mathbb{M}_{\mathbb{O}}:p_r(\mu^{(r)},\upsilon^{(r)})<\delta_1\}\,,\\
 \label{eq:22} \{\upsilon\in
  \mathbb{M}_{\mathbb{O}}:p_r(\mu^{(r)},\upsilon^{(r)})<\epsilon\}&\supset
  \{\upsilon\in
    \mathbb{M}_{\mathbb{O}}:d_{\mathbb{M}_{\mathbb{O}}}(\mu,\upsilon)<\delta_2\}\,.
\end{align}
For $s>r$ and $\mu,\upsilon\in
\mathbb{M}_{\mathbb{O}}$, $p_r(\mu^{(r)},\upsilon^{(r)})\ge
p_s(\mu^{(s)},\upsilon^{(s)})$. Let $y=\varepsilon/(2-\varepsilon)$. Choose
$\delta_1=y/(1-y)$ and $r=\log(1+y)$. Then, if $p_r(\mu^{(r)},\upsilon^{(r)})<\delta_1$
\begin{eqnarray*}
  d_{\mathbb{M}_{\mathbb{O}}}(\mu,\upsilon)\le \int_0^r e^{-x}\,dx+
  \int_r^{\infty}e^{-x}p_r(\mu^{(r)},\upsilon^{(r)})\,dx\le \varepsilon\,.
\end{eqnarray*}
Similarly, if $p_r(\mu^{(r)},\upsilon^{(r)})<\varepsilon$, we can
choose $\delta_2=1-e^{-r}(1+\varepsilon)^{-1}$ and \eqref{eq:22}
holds. According to \eqref{eq:11} and \eqref{eq:22}, it is enough to
consider the open sets in $(\mathbb{M}_b(\mathbb{S}\setminus \mathbb{C}^r), p_r)$
for $r>0$, which generates $\mathcal{M}_{\mathbb{O}}$. In what
follows, the connection between the functions $\Phi_A$ with $A\in
\mathcal{B}(\mathbb{S}\setminus \mathbb{C}^r)$ for some $r>0$ and the open sets in
$(\mathbb{M}_b(\mathbb{S}\setminus \mathbb{C}^r), p_r)$ will be explored. 

Let $F\subset \mathbb{S}\setminus \mathbb{C}^r$ be a closed
set. Choose $y>0$ and find $\mu\in \mathbb{M}_{\mathbb{O}}$ such that for some
$y>0$, $\mu(F)<y$ and let $\varepsilon=y-\mu(F)$. Recall that the
measures in $\mathbb{M}_b(\mathbb{S}\setminus \mathbb{C}^r)$ are
totally finite measures. According to Proposition A2.5.I,
\cite{Daley:2003tf}, the set 
\begin{equation}\label{eq:openset}
  D=\{\upsilon\in \mathbb{M}_{\mathbb{O}}: \Phi_F(\upsilon)<y\}=\{\upsilon\in
  \mathbb{M}_{\mathbb{O}}: \upsilon^{(r)}(F)<\mu^{(r)}(F)+\varepsilon\}
\end{equation} 
is an open set, that is, for any $\upsilon\in D$, there exists
$\delta>0$ such that
\begin{eqnarray*}
\{\widetilde{\upsilon}:p_r(\widetilde{\upsilon}^{(r)},\upsilon^{(r)})<\delta\}\subset
  D\,,
\end{eqnarray*}
and by \eqref{eq:22}, there exists $\delta^{\prime}_2>0$ such that 
\begin{eqnarray*}
  \{\widetilde{\upsilon}:d_{\mathbb{M}_{\mathbb{O}}}(\widetilde{\upsilon},
  \upsilon)<\delta^{\prime}_2\} \subset D\,.  
\end{eqnarray*}
Consequently, $D$ is an open set in $\mathbb{M}_{\mathbb{O}}$ and hence
measurable. 

Let $\mathcal{A}$ be the collection of sets $A\in
\mathcal{S}_{\mathbb{O}}$ for which $\Phi_A$ is
$\mathcal{M}_{\mathbb{O}}$-measurable. We will show that $\mathcal{A}$
agrees with $\mathcal{S}_{\mathbb{O}}$. Trivially, $\mathcal{A}\subset
\mathcal{S}_{\mathbb{O}}$. It remain to prove
$\mathcal{S}_{\mathbb{O}}\subset \mathcal{A}$. According to the
discussions above, $\mathcal{A}$ contains all the closed set $F$
bounded away from $\mathbb{C}$, that is, $F\subset \mathbb{S}\setminus
\mathbb{C}^r$ for some $r>0$. Moreover, for any $A,B\in
\mathcal{S}_{\mathbb{O}}$ bounded away from $\mathbb{C}$, the
intersection $A\cap B$ is also bounded away from $\mathbb{C}$. This
shows that the sets bounded away from $\mathbb{C}$ forms a
$\pi$-system. Moreover, $\mathcal{C}$ generates
$\mathcal{S}_{\mathbb{O}}$. By the definition of $\Phi$, $\Phi_{A\cup
  B}=\Phi_A+\Phi_B$ if $A$ and $B$ 
are disjoint and $\Phi_{A\setminus B}=\Phi_A-\Phi_B$ if $B\subset
A$. Since
\begin{align*}
A\cap B=(A\cup B)\setminus \big(((A\cup B)\setminus A)\cup (A\cup
  B)\setminus B\big)\,,
\end{align*}
we have $\Phi_{A\cap B}=\Phi_A+\Phi_B-\Phi_{A\cup B}$. This implies
that $\mathcal{D}\subset \mathcal{A}$. For $A\in
\mathcal{S}_{\mathbb{O}}$, let $A^{(r_n)}=A\setminus
\mathbb{C}^{r_n}$ with $r_n\downarrow 0$ as $n\to \infty$. The
sequence of functions $(\Phi_{A^{(r_n)}})$ is a 
non-decreasing sequence of measurable functions and for every $\mu\in
\mathbb{M}_{\mathbb{O}}$, $\lim_{n\to
  \infty}\Phi_{A^{(r_n)}}(\mu)=\Phi_A(\mu)$. Lebesgue's monotone
convergence theorem yields that $\Phi_A$ is also measurable. By
choosing $A$ as $\mathbb{O}$, we prove that $\Phi_{\mathbb{O}}$ is
measurable and consequently, $\mathbb{O}\in \mathcal{A}$. Furthermore,
if $A\in \mathcal{A}$, we have $\mathbb{O}\setminus A\in \mathcal{A}$
since $\Phi_{\mathbb{O}\setminus A}=\Phi_{\mathbb{O}}-\Phi_A$ is
measurable. Consider a
sequence of disjoint set $(A_i)_{i\ge 1}$ with $A_i\in
\mathcal{S}_{\mathbb{O}}$. Let $B_{n}=\bigcup_{i=1}^nA_i^{(r_n)}$ and
trivially, the sequence $\Phi_{B_n}$ is non-decreasing with
$\lim_{n\to \infty}\Phi_{B_n}=\Phi_{\cup_i A_i}$ pointwisely. An
application of Lebesgue's monotone convergence theorem ensures that
$\Phi_{\cup_i A_i}$ is measurable and hence, $\bigcup_iA_i\in
\mathcal{A}$. This shows that $\mathcal{A}$ is a
$\lambda$-system. According to Dynkin's $\pi-\lambda$ theorem, 
\begin{align*}
\mathcal{S}_{\mathbb{O}}=\sigma(\mathcal{D})\subset \mathcal{A}\,,
\end{align*}
which means that $\mathcal{S}_{\mathbb{O}}=\mathcal{A}$. 

It remains to show that $\mathcal{M}_{\mathbb{O}}$ is the
smallest $\sigma$-algebra in $\mathbb{M}_{\mathbb{O}}$ with this
property. Let $\mathcal{D}$ be given and let $\mathcal{G}$ be any
$\sigma$-algebra with respect to which $\Phi_A$ is measurable for all
$A\in \mathcal{A}$. It follows from the above arguments that $\Phi_A$
is $\mathcal{G}$-measurable for all $A\in
\sigma(\mathcal{D})=\mathcal{S}_{\mathbb{O}}$. Assume that $\mu\in
\mathbb{M}_{\mathbb{O}}$, a closed set $F\in \mathcal{S}_{\mathbb{O}}$
bounded away from $\mathbb{C}$ and $y>0$ are given. The set
$D$ defined in \eqref{eq:openset} is then a open set of
$\mathbb{M}_{\mathbb{O}}$ and an element of $\mathcal{G}$, which implies that
$\mathcal{G}$ contains a basis for the open sets of
$\mathcal{M}_{\mathbb{O}}$. Since 
$(\mathbb{M}_{\mathbb{O}},d_{\mathbb{M}_{\mathbb{O}}})$ is separable,
any open set of $\mathcal{M}_{\mathbb{O}}$ can be written as a
countable union of basic sets and thus, all open sets are in
$\mathcal{G}$. Therefore, $\mathcal{G}$ contains
$\mathcal{M}_{\mathbb{O}}$, which completes the proof. 
\end{proof}
\par
\begin{proof}[Proof of Lemma~\ref{lem:tight}]
  According to Theorem A2.4.I, \cite{Daley:2003tf} and Theorem 2.5,
  \cite{Lindskog:2014ia}, the set $K\in \mathcal{M}_{\mathbb{O}}$ is
  compact if there exists a sequence $(r_i)$ with $r_i\downarrow 0$
  such that for each $i$ and any $\varepsilon>0$, there exist constant $M_i>0$ and
  compact sets $K_{i,\varepsilon^{\prime}}\subset \mathbb{S}\setminus
  \mathbb{C}^{r_i}$ such that
\begin{align}
\label{eq:compact1}
\sup_{\xi\in K}\xi(\mathbb{S}\setminus \mathbb{C}^{r_i})&<M_i\,,\\
\label{eq:compact2}
\sup_{\xi\in K}\xi(\mathbb{S}\setminus (K_{i,\varepsilon^{\prime}}\cup
  \mathbb{C}^{r_i}))&\le \varepsilon\,.
\end{align}
Suppose that \eqref{eq:tight1} and \eqref{eq:tight2} are
satisfied. From \eqref{eq:tight1}, we choose $\widetilde{M}_i$ such
that 
\begin{equation*}
P_t \Big(\xi(\mathbb{S}\setminus \mathbb{C}^{r_i}) >\widetilde{M}_i\Big)<\varepsilon/2^{n+1}\,,
\end{equation*}
and from \eqref{eq:tight2}, we choose the compact set
$\widetilde{K}_{ij}\subset \mathbb{S}\setminus \mathbb{C}^{r_i}$ such
that 
\begin{equation*}
P_t \Big(\xi(\mathbb{S}\setminus (\widetilde{K}_{ij}\cup \mathbb{C}^{r_i}))>j^{-1} \Big)<\varepsilon/2^{m+n+2}\,.
\end{equation*}
Define the sets, for $i,j\ge 1$,
\begin{align*}
&Q_i=\{\xi:\xi(\mathbb{S}\setminus \mathbb{C}^{r_i})\le \widetilde{M}_i\}\,,\\ 
&Q_{i,j}=\{\xi:\xi(\mathbb{S}\setminus (\widetilde{K}_{ij}\cup
  \mathbb{C}^{r_i})\le j^{-1})\}\,.
\end{align*}
Let $K=\bigcap_{i=1}^{\infty}\bigcap_{j=1}^{\infty}(Q_i\cap Q_{i,j})$
and \eqref{eq:compact1} and \eqref{eq:compact2} are satisfied by
construction. Hence, $K$ is compact and 
\begin{equation*}
P_t(K^c)\le \sum_{i=1}^{\infty}
\Big[P_t(Q_i^c)+\sum_{j=1}^{\infty}P_t(Q_{i,j}^c) \Big]\le
\sum_{i=1}^{\infty}\Big[\frac{\varepsilon}{2^{n+1}}+\sum_{j=1}^{\infty}\frac{\varepsilon}{2^{m+n+2}}\Big]=\varepsilon\,. 
\end{equation*}
Thus, $(P_t)_{t\in \mathcal{T}}$ are uniformly tight. 

Suppose that the measures $(P_t)_{t\in \mathcal{T}}$ are uniformly
tight. For a given $\varepsilon>0$, a compact set $K\in
\mathcal{M}_{\mathbb{O}}$ exists and consequently, there exists a
sequence $(r_i)$ with $r_i\downarrow 0$ such that for each $i$ and any $\varepsilon>0$,
there exist constants $M_i>0$ and compact sets $K_{i,\varepsilon^{\prime}}\subset
\mathbb{S}\setminus \mathbb{C}^{r_i}$ such that for all $\xi\in K$,
\eqref{eq:compact1} and \eqref{eq:compact2} hold. Consequently,
\eqref{eq:tight1} and \eqref{eq:tight2} are satisfied. 
\end{proof}
\par
\begin{proof}[Proof of Theorem~\ref{thm:fidiconvergence}]
Lemma~\ref{lem:wc} have proved the first part of the theorem. Now we need to
prove the other part. Since a probability measure is uniquely determined
by the class of all fidi distributions, it suffices to show that the
family $(P_n)$ is uniformly tight, which implies that if the fidi
distributions of a subsequence of $(P_n)$ converges weakly to those of
$P$, then all convergent subsequence have the same limit and the whole
sequence $(P_n)$ converges weakly to $P$. 

We then need to check that \eqref{eq:tight1} and \eqref{eq:tight2} are
satisfied by using the assumption of the convergence of fidi
distributions. By Lemma~\ref{lem:fidifinite}, a sequence $(r_i)_{i\ge
  1}$ exists such that $r_i\downarrow 0$ and for all $i$,
$\mathbb{C}^{r_i}$ are stochastic continuity sets for all $P_n$, $n\ge
1$ and $P$. For a given $i$ and $\varepsilon>0$, we can find a constant $M>0$
that is continuity point for the distribution $\xi(\mathbb{S}\setminus
\mathbb{C}^{r_i})$ and for which $P(\xi(\mathbb{S}\setminus
\mathbb{C}^{r_i})>M)\le \varepsilon/2$ and $P_n(\xi(\mathbb{S}\setminus
\mathbb{C}^{r_i})>M)\to P(\xi(\mathbb{S}\setminus
\mathbb{C}^{r_i}))$. For some $N>0$ and $n>N$,
$P_n(\xi(\mathbb{S}\setminus \mathbb{C}^{r_i})>M)<\varepsilon$ and
then we can ensure that \eqref{eq:tight1} holds for all $n$ by
increasing $M$.

Since $(\mathbb{S},d)$ is complete and separable and the sets
$\mathbb{S}\setminus \mathbb{C}^{r_i}$, $i\ge 1$, are closed, the
subspaces $(\mathbb{S}\setminus \mathbb{C}^{r_i},d)$ are also complete
and separable. For each $i$, we can find a countable dense set
$D_i=\{x_j\,, j\ge 1\}$ of $\mathbb{S}\setminus \mathbb{C}^{r_i}$. By
Lemma~\ref{lem:fidifinite}, for any $j,l\ge 1$, we can find a
neighborhood of $x_j$, $B_{jl}=B_{r_l}(x_j)$ such that $r_j\le
\min\{d(x_j,\mathbb{O}),2^{-l}\}$ and $B_{jl}$ is a stochastic
continuity set for $P$. Trivially, 
\[\xi \Big(\bigcup_{j=1}^KB_{jl} \Big)\uparrow \xi(\mathbb{S}\setminus
  \mathbb{C}^{r_i})\,. \]
Given $\varepsilon,\varepsilon^{\prime}>0$, we can choose $K_l$ such that 
\begin{equation*}
P \Big(\xi(\mathbb{S}\setminus (\mathbb{C}^{r_i}\cup B_l))\ge \varepsilon^{\prime}_j
\Big)\le \varepsilon/2^{j+1}\,,
\end{equation*}
where $B_l=\cup_{j=1}^{K_l}B_{jl}$ and $\varepsilon^{\prime}_j\le
\varepsilon^{\prime}/2^j$ is chosen to be a continuity point of the
distribution of $\xi(\mathbb{S}\setminus (\mathbb{C}^{r_i}\cup B_l))$
under $P$. Using the weak convergence of the fidi distribution and
increasing the value of $K_l$ if necessary, we can ensure that for all
$n$,
\[P_n \Big(\xi(\mathbb{S}\setminus (\mathbb{C}^{r_i}\cup B_l))\ge
  \varepsilon^{\prime}_j \Big)\le \varepsilon/2^j\,. \]
Define $K_i=\bigcap_{l=1}^{\infty}B_{l}$. By construction, $K_{i}$ is closed
and it can be covered by a finite number of $\varepsilon$-spheres for
every $\varepsilon>0$. Consequently, $K_i$ is compact. Moreover, for
every $n$, 
\begin{align*}
P_n \Big(\xi(\mathbb{S}\setminus \mathbb{C}^{r_i})-\xi(K_i)>\varepsilon^{\prime}
  \Big)&=P_n \Big(\xi \Big(\bigcup_{l=1}^{\infty}(\mathbb{S}\setminus
         (\mathbb{C}^{r_i}\cup B_l)) \Big)>\varepsilon^{\prime} \Big)\\
&\le \sum_{l=1}^{\infty}P_n \Big(\xi(\mathbb{S}\setminus
         (\mathbb{C}^{r_i}\cup B_l))>\varepsilon^{\prime}/2^j \Big)\\ 
&\le \sum_{l=1}^{\infty}P_n \Big(\xi(\mathbb{S}\setminus
         (\mathbb{C}^{r_i}\cup B_l))>\varepsilon^{\prime}_j \Big)\\ 
&\le \sum_{j=1}^{\infty}\frac{\varepsilon}{2^j}=\varepsilon\,,  
\end{align*}
and thus \eqref{eq:tight2} holds. This completes the proof. 
\end{proof}
\par
\begin{proof}[Proof of Lemma~\ref{lem:310}]
  Let 
\[\mathcal{A}=\{f:f\in BM(\mathbb{\widetilde{S}}),\; f \text{
    satisfies \eqref{eq:prod1}}\}\,.\]
Think of the functions of the form $f(y,x)=f_1(y)f_2(x)\in [0,1]$
where $f_1$ is non-negative and $\mathcal{K}$-measurable and $f_2\in
BM(\mathbb{S})$. An application of \eqref{eq:condindep} yields
that
\begin{eqnarray*}
E\Big(f(K_i,X_i)\mid (X_n)_{n\ge 1}\Big) 
&= & E\Big(f_1(K_i)f_1(X_i)\mid (X_n)_{n\ge 1} \Big)\\
& =&f_2(X_i)E  \Big(f_1(K_i)\mid (X_n)_{n\ge 1}\Big)\\ 
 &= & f_2(X_i) \int_{\mathbb{K}}f_1(y)\,G(dy, X_i)\\ 
 & =& \int_{\mathbb{K}}g(y,X_i)\,G(dy,X_i)\,.
\end{eqnarray*}
This implies that the indicator function
$\mathbf{1}_{F_i\times A_i}=\mathbf{1}_{F_i}\mathbf{1}_{A_i}$ belongs
to $\mathcal{A}$, where 
$A_i\in \mathcal{S}_{\mathbb{O}}$ bounded away from $\mathbb{C}$ and
$F_i\in \mathcal{K}$. 

Fix $r>0$. Let $BM(\mathbb{\widetilde{S}}\setminus \mathbb{\widetilde{C}}^r)$ be a collection of functions $g\in \mathcal{A}$
such that $g$ vanishes on $\mathbb{\widetilde{C}}^r$ and let
$\mathcal{E}^r=\{G\in \mathcal{\widetilde{S}}_{\mathbb{\widetilde{O}}}:\mathbf{1}_G\in
BM(\mathbb{\widetilde{S}}\setminus
\mathbb{\widetilde{C}}^r)\}$. Following the arguments
in the proof of Proposition~\ref{prop:algebra}, 
it is trivial to show that $\mathcal{E}^r$ is a $\lambda$-system
containing the $\pi$-system of rectangles $F\times A$ with $A\in
\mathcal{B}(\mathbb{S}\setminus \mathbb{C}^r)$ and $F\in
\mathcal{K}$. Dynkin's theorem shows that $\mathcal{E}^r\supset
\mathcal{B}(\mathbb{\widetilde{S}}\setminus
\mathbb{\widetilde{C}}^r)$. Moreover, the simple function
\begin{equation}\label{eq:sf}
\sum_{i=1}^kc_i\mathbf{1}_{\widetilde{F}_i}\in BM(\mathbb{\widetilde{S}}\setminus \mathbb{\widetilde{C}}^r)\,,
\end{equation}
where $c_i\in [0,1]$ and $\widetilde{F}_i\in \mathcal{B}(\mathbb{S}\setminus
\mathbb{C}^r)\times \mathcal{K}$, $i=1,\ldots, k$. Any measurable function
$f:\mathbb{\widetilde{S}}\setminus \mathbb{\widetilde{C}}^r\to [0,1]$
in $BM(\mathbb{\widetilde{S}}\setminus \mathbb{\widetilde{C}}^r)$ can be
written as the monotone limit of simple functions of the form
\eqref{eq:sf} and Lebesgue's monotone convergence theorem yields that $f\in
\mathcal{A}$. Since $r$ is arbitrary, if $g\in
BM(\mathbb{\widetilde{S}})$, there exists $r>0$ such that $g$ vanishes
on $\mathbb{\widetilde{C}}^r$ and consequently, $g\in \mathcal{A}$. 

For $g\in \mathcal{A}$ of the form $g(y,x)=g_1(y)g_2(x)$, it is easy
to verify that \eqref{eq:prod2} is satisfied. Using similar arguments
as above, it is trivial to show that \eqref{eq:prod2} holds. 
\end{proof}

\bibliographystyle{amsplain}
\bibliography{B-RV}

\end{document}